\documentclass[12pt]{amsart}
\usepackage{amsmath,amsthm,amssymb}
\usepackage{graphicx}
\usepackage{subfig}

\DeclareMathOperator{\RE}{Re} \DeclareMathOperator{\IM}{Im}

\numberwithin{equation}{section}
\newtheorem{theorem}{Theorem}[section]
\newtheorem{lemma}[theorem]{Lemma}
\newtheorem{corollary}[theorem]{Corollary}
\theoremstyle{remark}
\newtheorem{remark}[theorem]{Remark}
\newtheorem{example}[theorem]{Example}
\newtheorem{definition}[theorem]{Definition}

\begin{document}

\title{Fully Starlike and Convex Harmonic Mappings of order $\alpha$}

\thanks{The research work of the first author is supported by research fellowship from Council of Scientific and
Industrial Research (CSIR), New Delhi.}

\author[S. Nagpal]{Sumit Nagpal}
\address{Department of Mathematics, University of Delhi,
Delhi--110 007, India}
\email{sumitnagpal.du@gmail.com }

\author[V. Ravichandran]{V. Ravichandran}

\address{Department of Mathematics, University of Delhi,
Delhi--110 007, India \and
School of Mathematical Sciences,
Universiti Sains Malaysia, 11800 USM, Penang, Malaysia}
\email{vravi68@gmail.com}

\begin{abstract}
The hereditary property of convexity and starlikeness for conformal mappings does not generalize to univalent harmonic mappings. This failure leads us to the notion of fully starlike and convex mappings of order $\alpha$, $(0\leq \alpha<1)$. A bound for the radius of fully starlikeness and fully convexity of order $\alpha$ is determined for certain families of univalent harmonic mappings. Convexity is not preserved under the convolution of univalent harmonic convex mappings, unlike in the analytic case. Given two univalent harmonic convex mappings $f$ and $g$, the problem of finding the radius $r_{0}$ such that $f*g$ is a univalent harmonic convex mapping in $|z|<r_{0}$, is being considered.
\end{abstract}



 \maketitle

\section{Introduction}
Let $\mathcal{H}$ denote the class of all complex-valued harmonic functions $f$ in the unit disk $\mathbb{D}=\{z \in \mathbb{C}:|z|<1\}$ normalized by $f(0)=0=f_{z}(0)-1$. Let $\mathcal{S}_{H}$ be the subclass of $\mathcal{H}$ consisting of univalent and sense-preserving functions. Such functions can be written in the form $f=h+\bar{g}$, where
\begin{equation}\label{eq1.1}
h(z)=z+\sum_{n=2}^{\infty}a_{n}z^{n}\quad\mbox{and}\quad g(z)=\sum_{n=1}^{\infty}b_{n}z^n
\end{equation}
are analytic and $|g'(z)|<|h'(z)| $ in $\mathbb{D}$. It follows that $|b_{1}|<1$ and hence the function $(f-\overline{b_{1}f})/(1-|b_{1}|^2)$ belongs to $\mathcal{S}_{H}$. Thus we may restrict our attention to the subclass
\[\mathcal{S}_{H}^{0}=\{f \in \mathcal{S}_{H}:f_{\bar{z}}(0)=0\}.\]
Observe that $\mathcal{S}_{H}$ reduces to $\mathcal{S}$, the class of normalized univalent analytic functions, if the co-analytic part of $f$ is zero. In $1984$, Clunie and Sheil-Small (see \cite{cluniesheilsmall}) investigated the class $\mathcal{S}_{H}$ as well as its geometric subclasses and obtained some coefficient bounds. Since then, there have been several related papers on $\mathcal{S}_{H}$ and its subclasses.

Let $\mathcal{S}_{H}^{*}$, $\mathcal{K}_{H}$ and $\mathcal{C}_{H}$ be the subclasses of $\mathcal{S}_{H}$ mapping $\mathbb{D}$ onto starlike, convex and close-to-convex domains, respectively, just as $\mathcal{S}^{*}$, $\mathcal{K}$ and $\mathcal{C}$ are the subclasses of $\mathcal{S}$ mapping $\mathbb{D}$ onto their respective domains. Denote by $\mathcal{S}_{H}^{*0}$, $\mathcal{K}_{H}^{0}$ and $\mathcal{C}_{H}^{0}$, the class consisting of those functions $f$ in $\mathcal{S}_{H}^{*}$, $\mathcal{K}_{H}$ and $\mathcal{C}_{H}$ respectively, for which $f_{\bar{z}}(0)=0$.

In \cite{cluniesheilsmall}, Clunie and Sheil-Small conjectured that if $f=h+\bar{g} \in \mathcal{S}_{H}^{0}$ then the Taylor coefficients of the series of $h$ and $g$ satisfy the inequality
\begin{equation}\label{eq1.2}
|a_{n}|\leq\frac{1}{6}(2n+1)(n+1)\quad\mbox{and}\quad|b_{n}|\leq \frac{1}{6}(2n-1)(n-1),\quad \mbox{for all }n \geq 1.
\end{equation}
They verified this conjecture for typically real functions. Later, Sheil-Small \cite{sheilsmall} proved it for all functions $f \in \mathcal{S}_{H}^{0}$ for which $f(\mathbb{D})$ is starlike with respect to the origin or $f(\mathbb{D})$ is convex in one direction. In \cite{wang}, Wang, Liang  and Zhang verified this conjecture for close-to-convex functions in $\mathcal{S}_{H}^{0}$. However, this coefficient conjecture remains an open problem for the full class $\mathcal{S}_{H}^{0}$. Equality occurs in \eqref{eq1.2} for the harmonic Koebe function
\begin{equation}\label{eq1.3}
K(z)=\frac{z-\frac{1}{2}z^2+\frac{1}{6}z^3}{(1-z)^3}+\overline{\frac{\frac{1}{2}z^2+\frac{1}{6}z^3}{(1-z)^3}},
\end{equation}
constructed by shearing the Koebe function $k(z)=z/(1-z)^2$ horizontally with dilatation $w(z)=z$. Note that $K$ maps the unit disk $\mathbb{D}$ onto the slit-plane $\mathbb{C}\backslash(-\infty,-1/6]$.

Regarding the coefficient bounds for functions in $\mathcal{S}_{H}^{0}$ mapping $\mathbb{D}$ onto a convex domain, Clunie and Sheil-Small \cite{cluniesheilsmall} proved that the Taylor coefficients of the series of $h$ and $g$ of a function $f \in \mathcal{K}_{H}^{0}$ satisfy the inequalities
\begin{equation}\label{eq1.4}
|a_{n}|\leq\frac{n+1}{2}\quad\mbox{and}\quad|b_{n}|\leq\frac{n-1}{2},\quad \mbox{for all }n \geq 1.
\end{equation}
Equality occurs for the harmonic half-plane mapping
\begin{equation}\label{eq1.5}
L(z)=\RE \left(\frac{z}{1-z}\right)+i \IM\left(\frac{z}{(1-z)^2}\right),
\end{equation}
constructed by shearing the conformal mapping $l(z)=z/(1-z)$ vertically with dilatation $w(z)=-z$.

It's worth to recall that convexity is a hereditary property for conformal mappings. In other words, if an analytic function maps the unit disk univalently onto a convex domain, then it also maps each concentric subdisk onto a convex domain. However, this hereditary property does not generalize to harmonic mappings. The harmonic half-plane mapping $L$ given by \eqref{eq1.5} sends the subdisk $|z|<r$ onto a convex region for $r \leq \sqrt{2}-1$, but onto a nonconvex region for $\sqrt{2}-1<r<1$. In fact, it has been proved that if a function $f$ maps the unit disk harmonically onto a convex domain, then for each radius $r \leq \sqrt{2}-1$ it again maps the disk $|z|<r$ onto a convex domain, but it need not do so for any radius in the interval $\sqrt{2}-1<r<1$ (see \cite{ruscheweyh,sheilsmall}).

In the same sense, starlikeness is a hereditary property for conformal mappings which does not generalize to harmonic mappings. This is seen by the following example.

\begin{example}\label{ex1.1}
Consider the harmonic half-plane mapping $L$ defined by \eqref{eq1.5}. Note that $L$ maps the unit disk $\mathbb{D}$ onto the half-plane $\RE \{w\}>-1/2$. We shall show that $L$ sends each disk $|z|<r\leq r_{0}$ to a starlike region, but the image is not starlike when $r_{0}<r<1$, where $r_{0}$ is given by
\[r_{0}=\sqrt{\frac{7\sqrt{7}-17}{2}}\approx0.871854.\]

We employ a similar calculation carried out in \cite[Section 3.5]{duren} to determine the radius of convexity for $L$. For this, it will be customary to study the change of the direction $\Psi_{r}(\theta)=\arg L(r e^{i\theta})$ of the image curve as the point $z=r e^{i\theta}$ moves around the circle $|z|=r$. A direct calculation gives
\[L(r e^{i\theta})=A(r,\theta)+iB(r,\theta),\]
where
\[|1-z|^2A(r,\theta)=r(\cos\theta-r)\quad \mbox{and}\quad |1-z|^4B(r,\theta)=r(1-r^2)\sin\theta.\]

\begin{figure}[hb]
\centering
\includegraphics[width=2.5in]{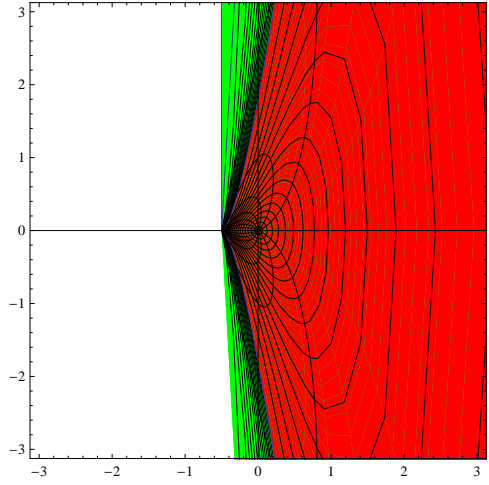}
\caption{Image of the subdisk $|z|<\sqrt{(7\sqrt{7}-17)/2}$ under the mapping $L$.}\label{fig1}
\end{figure}

The problem now reduces to finding the values of $r$ such that the argument of $L(r e^{i\theta})$ is a nondecreasing function of $\theta$. Writing $L(z)=h(z)+\overline{g(z)}$, observe that
\[\left.\frac{\partial}{\partial \theta}\Psi_{r}(\theta)\right|_{\cos \theta =r}=\left.\RE\left(\frac{zh'(z)-\overline{zg'(z)}}{h(z)+\overline{g(z)}}\right)\right|_{\cos \theta =r}=1.\]
If $\cos \theta \neq r$, then we shall show that
\[\tan \Psi_{r}(\theta)=\frac{B(r,\theta)}{A(r,\theta)}=\frac{(1-r^2)\sin\theta}{(1-2r\cos\theta+r^2)(\cos\theta-r)}\]
is a nondecreasing function of $\theta$. A straightforward calculation leads to an expression for the derivative in the form
\[(1-2ur+r^2)^2(u-r)^2 \frac{\partial}{\partial \theta}\tan \Psi_{r}(\theta)=(1-r^2)p(r,u),\]
where $u=\cos\theta$ and
\[p(r,u)=1+u(2u^2-5)r+3r^2-ur^3.\]

The problem is now to find the values of the parameter $r$ for which the polynomial $p(r,u)$ is non-negative in the whole interval $-1\leq u \leq 1$. Observe that
\[p(r,-1)=(1+r)^3>0\quad\mbox{and}\quad p(r,1)=(1-r)^3>0.\]
Also, differentiation gives
\[\frac{\partial}{\partial u}p(r,u)=(6u^2-5)r-r^3,\]
showing that $p(r,u)$ has a local minimum at $u=\sqrt{(5+r^2)/6}$ and a local maximum at $u=-\sqrt{(5+r^2)/6}$. Thus $p(r,u)\geq 0$ for $-1\leq u\leq 1$ if and only if
\[p\left(r,\sqrt{\frac{5+r^2}{6}}\right)=\frac{1}{9}[9-5\sqrt{6}r\sqrt{5+r^2}+27r^2-\sqrt{6}r^3\sqrt{5+r^2}] \geq 0.\]
This inequality implies that $r\leq r_{0}=\sqrt{(7\sqrt{7}-17)/2}$. This proves that the angle $\Psi_{r}(\theta)$ increases monotonically with $\theta$ if $r\leq r_{0}$ and hence the harmonic half-plane mapping $L$ sends each disk $|z|<r \leq r_{0}$ to a starlike region, but the image is not starlike when $r_{0}<r<1$ (see Figure \ref{fig1}).
\end{example}

The failure of the hereditary property for starlike and convex harmonic mappings led to the notion of fully starlike and fully convex functions, which is being discussed in \cite{chuaqui}. For $0 \leq\alpha <1$, the concept of fully starlike functions of order $\alpha$ and fully convex functions of order $\alpha$ is introduced in section \ref{sec2}, analogous to the subclasses $\mathcal{S}^{*}(\alpha)$ and $\mathcal{K}(\alpha)$ of $\mathcal{S}$, in the analytic case, consisting of respectively starlike functions of order $\alpha$ and convex functions of order $\alpha$. Recall that these classes are defined analytically by the equivalence
\begin{equation}\label{eq1.6}
f \in \mathcal{S}^{*}(\alpha) \Leftrightarrow \RE\left(\frac{zf'(z)}{f(z)}\right)>\alpha\quad \mbox{and}\quad f \in \mathcal{K}(\alpha) \Leftrightarrow \RE\left(\frac{zf''(z)}{f'(z)}+1\right)>\alpha.
\end{equation}

Let $\mathcal{F}$ be the family of all functions of the form $f=h+\overline{g}$ where $h$ and $g$ are given by \eqref{eq1.1}. In \cite{ponnusamy}, it has been proved that the radius of univalence and starlikeness of the family $\mathcal{F}$ is given by
\[r_{0}=1+\frac{1}{2\sqrt{2}}-\sqrt{\sqrt{2}+\frac{1}{8}}\approx0.112903,\]
if the coefficients of the series satisfy conditions \eqref{eq1.2}. A similar calculation is carried out if the coefficients of the series satisfy \eqref{eq1.4}. In this case, the radius of univalence and starlikeness of $\mathcal{F}$ comes out to be $s_{0}$, given by
\[s_{0}=1+\frac{\sqrt[3]{-18+\sqrt{330}}}{6^{2/3}}-\frac{1}{\sqrt[3]{6(-18+\sqrt{330})}}\approx0.164878.\]
These results are generalized in context of fully starlike and fully convex functions of order $\alpha$ $(0<\alpha\leq 1)$ in section \ref{sec3}. The results, in turn, provide a bound for the radius of fully starlikeness (resp. fully convexity) of order $\alpha$, for the convex, starlike and close-to-convex mappings in $S_{H}$.

For analytic functions $f(z)=z+\sum_{n=2}^{\infty}a_{n}z^{n}$ and $F(z)=z+\sum_{n=2}^{\infty}A_{n}z^{n}$, their convolution (or Hadamard product) is defined as
\[f*F=z+\sum_{n=2}^{\infty}a_{n}A_{n}z^{n}.\]
In the harmonic case, with
\begin{align*}
f&=h+\bar{g}=z+\sum_{n=2}^{\infty}a_{n}z^{n}+\overline{\sum_{n=1}^{\infty}b_{n}z^{n}},\quad \mbox{and}\\
F&=H+\bar{G}=z+\sum_{n=2}^{\infty}A_{n}z^{n}+\overline{\sum_{n=1}^{\infty}B_{n}z^{n}}.
\end{align*}
their harmonic convolution is defined as
\[f*F=h*H+\overline{g*G}=z+\sum_{n=2}^{\infty}a_{n}A_{n}z^{n}+\overline{\sum_{n=1}^{\infty}b_{n}B_{n}z^{n}}.\]
There have been many results about harmonic convolutions (see \cite{cluniesheilsmall,dorff1,dorff2,goodloe,ruscheweyh}). For the convolution of analytic functions, if $f_{1}$, $f_{2} \in \mathcal{K}$, then $f_{1}*f_{2} \in \mathcal{K}$. However, it is easy to see that the Hadamard product of two functions in $\mathcal{K}_{H}$ is not necessarily convex, or even univalent. In Section \ref{sec4}, the radius of univalence of the family
\[\mathcal{G}=\left\{f=h+\bar{g} \in \mathcal{H}:|a_{n}|\leq\left(\frac{n+1}{2}\right)^{2}\, \mbox{ and }\, |b_{n}|\leq\left(\frac{n-1}{2}\right)^{2}\, \mbox{ for } \, n\geq 1\right\}\]
is determined, which turns out to be $r_{0}\approx 0.129831$. This number is also the radius of starlikeness of $\mathcal{G}$. The radius of convexity of $\mathcal{G}$ is $s_{0}\approx0.0712543$. This, in particular, shows that if $f$, $g \in \mathcal{K}_{H}^{0}$, then $f*g$ is univalent and convex in at least $|z|<s_{0}\approx 0.0712543$.


\section{Fully Starlikeness and Convexity of order $\alpha$}\label{sec2}
 In this section, we will introduce the concept of fully starlike functions of order $\alpha$ $(0\leq \alpha <1)$ and fully convex functions of order $\alpha$  and give some of their basic properties with illustrations.

\begin{definition}\label{def2.1}
A harmonic mapping $f$ of the unit disk $\mathbb{D}$ is said to be \textbf{fully convex of order $\alpha$} $(0\leq \alpha <1)$ if it maps every circle $|z|=r<1$ in a one-to-one manner onto a convex curve satisfying
\begin{equation}\label{eq2.1}
\frac{\partial}{\partial \theta}\left(\arg \left\{\frac{\partial}{\partial \theta}f(r e^{i \theta})\right\}\right)> \alpha,\quad 0\leq\theta\leq 2\pi,\quad 0<r<1.
\end{equation}
If $\alpha=0$, then $f$ is said to be \emph{fully convex}.
\end{definition}

According to the Rad\'{o}-Kneser-Choquet theorem, a fully convex harmonic mapping of order $\alpha$ $(0\leq\alpha<1)$ is necessarily univalent in $\mathbb{D}$. The affine mappings $f(z)=\alpha z+\gamma+\beta \overline{z}$ with $|\alpha|>|\beta|$, are fully convex of order $(|\alpha|-|\beta|)/(|\alpha|+|\beta|)$. If $f \in \mathcal{K}_{H}$, then $f$ is fully convex in $|z|<\sqrt{2}-1$ with the extremal function as the harmonic half-plane mapping $L$ defined by \eqref{eq1.5} (see \cite{ruscheweyh}). Similarly, $f$ is fully convex in $|z|<3-\sqrt{8}$ if $f \in \mathcal{S}_{H}^{*}$ or $\mathcal{C}_{H}$ and the harmonic Koebe function $K$ given by \eqref{eq1.3} shows that this bound is best possible (see \cite{sheilsmall}). However, the exact radius of fully convexity of order $\alpha$ ($0<\alpha<1$) for starlike, convex and close-to-convex mappings in $\mathcal{S}_{H}$ is still unsolved (see Section \ref{sec3}).

It's a worth to remark that the condition \eqref{eq2.1} is sufficient but not necessary for a function $f \in \mathcal{S}_{H}$ to map $\mathbb{D}$ onto a convex domain (see \cite[Theorem 3]{mocanu}). The next theorem provides a sufficient condition for a sense-preserving harmonic mapping to be fully convex.
\begin{theorem}\label{convex}
A sense-preserving harmonic function $f=h+\bar{g}$ is fully convex in $\mathbb{D}$ if the analytic functions $h+\epsilon g$ are convex in $\mathbb{D}$ for each $|\epsilon|=1$.
\end{theorem}

\begin{proof}
To prove the assertion, it suffices to show that $f$ is convex in $|z|<r$ for each $r\leq 1$. To see this, fix $r_0 \in (0,1]$. Then the analytic functions $h+\epsilon g$ are convex in $|z|<r_0$. By \cite[Theorem 5.7]{cluniesheilsmall}, it follows that $f$ is convex in $|z|<r_0$.
\end{proof}

However, if $f=h+\bar{g}$ is fully convex then the functions $h+\epsilon g$ need not be convex for each $|\epsilon|=1$. For this, consider the function $F(z)=L((\sqrt{2}-1)z)$, $z \in \mathbb{D}$, where $L$ is given by \eqref{eq1.5}. Writing $F=H+\bar{G}$, we see that $H-G=k((\sqrt{2}-1)z)$ which is not convex, $k$ being the Koebe function.

Let $\mathcal{FK}_{H}(\alpha)$ $(0\leq\alpha<1)$ denote the subclass of $\mathcal{K}_{H}$ consisting of fully convex functions of order $\alpha$, with $\mathcal{FK}_{H}:=\mathcal{FK}_{H}(0)$ and let $\mathcal{FK}_{H}^{0}(\alpha)=\mathcal{FK}_{H}(\alpha) \cap \mathcal{K}_{H}^{0}$. In terms of the coefficients, Jahangiri \cite{jahangiriconvex} gave a sufficient condition for functions $f \in \mathcal{H}$ to be in $\mathcal{FK}_{H}(\alpha)$.

\begin{lemma}\cite{jahangiriconvex}\label{lem2.2}
Let $f=h+\bar{g}$, where $h$ and $g$ are given by \eqref{eq1.1}. Furthermore, let
\[\sum_{n=2}^{\infty}\frac{n(n-\alpha)}{1-\alpha}|a_{n}|+\sum_{n=1}^{\infty}\frac{n(n+\alpha)}{1-\alpha}|b_{n}|\leq 1\]
and $0\leq \alpha<1$. Then $f \in \mathcal{FK}_{H}(\alpha)$.
\end{lemma}

The analytic description of functions in $\mathcal{FK}_{H}(\alpha)$ $(0\leq\alpha<1)$ is seen by the following theorem.

\begin{theorem}\label{th2.3}
Let $f=h+\bar{g} \in \mathcal{H}$ be sense-preserving and let $0\leq\alpha < 1$. Then $f \in \mathcal{FK}_{H}(\alpha)$ if and only if
\begin{equation}\label{eq2.2}
\begin{split}
|zh'(z)|^{2}\left[\RE \left(1+\frac{zh''(z)}{h'(z)}\right)-\alpha\right]&>|zg'(z)|^{2}\left[\RE \left(1+\frac{zg''(z)}{g'(z)}\right)+\alpha\right]\\
&+\RE[z^2(zh''(z)g'(z)-2\alpha h'(z)g'(z)-zh'(z)g''(z))]
\end{split}
\end{equation}
for all $z \in \mathbb{D}$.
\end{theorem}

\begin{proof}
Suppose that $f \in  \mathcal{FK}_{H}(\alpha)$. A simple calculation shows that \eqref{eq2.1} reduces to the condition \eqref{eq2.2}. Conversely, if $f$ satisfies \eqref{eq2.2}, then by the proof of Theorem 3 in \cite[p.\ 139]{chuaqui}, $f$ maps each circle $|z|=r<1$ in a one-to-manner onto a convex curve satisfying \eqref{eq2.1} so that $f \in  \mathcal{FK}_{H}(\alpha)$.
\end{proof}

If $\alpha=0$, then Theorem \ref{th2.3} reduces to \cite[Theorem 3, p. 139]{chuaqui}. We next define the notion of fully starlike functions of order $\alpha$ $(0\leq\alpha<1)$.
\begin{definition}\label{def2.4}
A harmonic mapping $f$ of the unit disk $\mathbb{D}$ with $f(0)=0$ is said to be \textbf{fully starlike of order $\alpha$} $(0\leq \alpha <1)$ if it maps every circle $|z|=r<1$ in a one-to-one manner onto a curve that bounds a domain starlike with respect to the origin satisfying
\begin{equation}\label{eq2.3}
\frac{\partial}{\partial \theta}\arg f(r e^{i \theta})> \alpha,\quad 0\leq\theta\leq 2\pi,\quad 0<r<1.
\end{equation}
If $\alpha=0$, then $f$ is said to be \emph{fully starlike}.
\end{definition}

Example \ref{ex1.1} shows that the harmonic half-plane mapping $L$ given by \eqref{eq1.5} is fully starlike in $|z|<\sqrt{(7\sqrt{7}-17)/2}\approx0.871854$. The affine mappings $f(z)=\alpha z+\beta \overline{z}$ with $|\alpha|>|\beta|$, are fully starlike of order $(|\alpha|-|\beta|)/(|\alpha|+|\beta|)$.
Unlike fully convex mappings, a fully starlike mapping need not be univalent (see \cite{chuaqui}). It is also clear that every fully convex mapping of order $\alpha$ is fully starlike of order $\alpha$. However, the converse is not true as seen by the following example.

\begin{example}\label{ex2.5}
For $n \geq 2$ and $\alpha \in[0,1)$, consider the function
\[f_{n}(z)=z+\frac{1-\alpha}{n+\alpha}\bar{z}^n,\quad z\in \mathbb{D}.\]
The functions $f_{n}$ $(n\geq 2)$ are fully starlike of order $\alpha$. In fact, for $z=r e^{i\theta}$, we have
\[\frac{\partial}{\partial \theta}\arg f_{n}(r e^{i \theta})=\RE  \left(\frac{z-\frac{n(1-\alpha)}{n+\alpha}\bar{z}^n}{z+\frac{1-\alpha}{n+\alpha}\bar{z}^n}\right) \geq\frac{1-\frac{n(1-\alpha)}{n+\alpha}r^{n-1}}{1+\frac{1-\alpha}{n+\alpha}r^{n-1}}>\alpha.\]
Further, observe that
\[\frac{\partial}{\partial \theta}\left(\arg \left\{\frac{\partial}{\partial \theta}f_{n}(r e^{i \theta})\right\}\right)=\RE \frac{z+\frac{n^2(1-\alpha)}{n+\alpha}\bar{z}^n}{z-\frac{n(1-\alpha)}{n+\alpha}\bar{z}^n}\geq\frac{1-\frac{n^2(1-\alpha)}{n+\alpha}r^{n-1}}{1+\frac{n(1-\alpha)}{n+\alpha}r^{n-1}}.\]
Therefore, it follows that $f_{n}$ is fully convex of order $\alpha$ in $|z|<1/n^{1/(n-1)}$. In particular, this shows that $f_{n}$ is not fully convex of order $\alpha$ in $\mathbb{D}$.
\end{example}

The condition \eqref{eq2.3} is sufficient but not necessary for a function $f \in \mathcal{S}_{H}$ to map $\mathbb{D}$ onto a starlike domain (see \cite[Theorem 1]{mocanu}). Similar to Theorem \ref{convex}, the next theorem provides a sufficient condition which guarantees a sense-preserving harmonic mapping to be fully starlike. It's proof follows by invoking \cite[Theorem 3, p. 10]{stable}.

\begin{theorem}\label{starlike}
A sense-preserving harmonic function $f=h+\bar{g}$ is fully starlike in $\mathbb{D}$ if the analytic functions $h+\epsilon g$ are starlike in $\mathbb{D}$ for each $|\epsilon|=1$.
\end{theorem}

Let $\mathcal{FS}^{*}_{H}(\alpha)$ denote the subclass of $\mathcal{S}^{*}_{H}$ consisting of fully starlike functions of order $\alpha$ $(0\leq \alpha <1)$, with $\mathcal{FS}^{*}_{H}:=\mathcal{FS}^{*}_{H}(0)$ and let $\mathcal{FS}_{H}^{*0}(\alpha)=\mathcal{FS}_{H}^{*}(\alpha) \cap \mathcal{S}_{H}^{*0}$. In \cite{jahangiristarlike}, Jahangiri gave a sufficient condition for functions $f \in \mathcal{H}$ to be in $\mathcal{FS}^{*}_{H}(\alpha)$.

\begin{lemma}\cite{jahangiristarlike}\label{lem2.6}
Let $f=h+\bar{g}$, where $h$ and $g$ are given by \eqref{eq1.1}. Furthermore, let
\[\sum_{n=2}^{\infty}\frac{n-\alpha}{1-\alpha}|a_{n}|+\sum_{n=1}^{\infty}\frac{n+\alpha}{1-\alpha}|b_{n}|\leq 1\]
and $0\leq \alpha<1$. Then $f \in \mathcal{FS}^{*}_{H}(\alpha)$.
\end{lemma}

Corresponding to Theorem \ref{th2.3}, the analytic characterization of functions in  $\mathcal{FS}^{*}_{H}(\alpha)$ is given in the following theorem with the case $\alpha=0$ being reduced to \cite[Theorem 3, p. 139]{chuaqui}.

\begin{theorem}\label{th2.7}
Let $f=h+\bar{g} \in \mathcal{H}$ be sense-preserving and let $0\leq\alpha < 1$. Then $f \in \mathcal{FS}^{*}_{H}(\alpha)$ if and only if $f(z)\neq 0$ for $0<|z|<1$ and
\begin{equation}\label{eq2.4}
|h(z)|^{2}\left(\RE \frac{zh'(z)}{h(z)}-\alpha\right)>|g(z)|^{2}\left(\RE \frac{zg'(z)}{g(z)}+\alpha\right)+\RE (zh(z)g'(z)+2\alpha h(z) g(z)-zh'(z)g(z))
\end{equation}
for all $z \in \mathbb{D}$.
\end{theorem}

\begin{proof}
The necessary part follows immediately by the univalence of $f$ and \eqref{eq2.3}. For sufficiency, suppose that $f(z) \neq 0$ for $z \neq 0$ and \eqref{eq2.4} holds. Then, by the proof of Theorem 3 in \cite[p.\ 139]{chuaqui}, $f$ is fully starlike in $\mathbb{D}$ satisfying \eqref{eq2.3} and hence univalent by \cite[Theorem 1]{mocanu}, so $f \in \mathcal{FS}^{*}_{H}(\alpha)$.
\end{proof}

\begin{remark}\label{rem2.8}
For normalized analytic functions $f$, conditions \eqref{eq2.2} and \eqref{eq2.4} reduces to \eqref{eq1.6}. In general, given $\alpha \in [0,1)$ and $h \in \mathcal{K}(\alpha)$, the function
\[f=h+\epsilon \bar{h} \in \mathcal{FK}_{H}\left(\frac{1-|\epsilon|}{1+|\epsilon|}\alpha\right)\quad \mbox{for}\quad |\epsilon|<1.\]
To see this, note that for $z=r e^{i\theta} \in \mathbb{D}\backslash\{0\}$, we have
\begin{align*}
\frac{\partial}{\partial \theta}\left(\arg \left\{\frac{\partial}{\partial \theta}f(r e^{i\theta})\right\}\right)&=\RE\left(\frac{zh'(z)+z^{2}h''(z)+\overline{\epsilon(zh'(z)+z^2 h''(z))}}{zh'(z)-\overline{\epsilon z h'(z)}}\right)\\
                          &=\frac{(1-|\epsilon|^{2})|zh'(z)|^{2}}{|zh'(z)-\overline{\epsilon z h'(z)}|^{2}}\RE \left(1+\frac{zh''(z)}{h'(z)}\right)\\
                          &>\frac{1-|\epsilon|}{1+|\epsilon|}\alpha
\end{align*}
In particular, this shows that $\mathcal{K}(\alpha)\subset\mathcal{FK}_{H}(\alpha)$. A similar statement shows that $\mathcal{S}^{*}(\alpha)\subset\mathcal{FS}^{*}_{H}(\alpha)$
\end{remark}

The next example examines the fully convexity and fully starlikeness of order $\alpha$ ($0 \leq \alpha <1$) of the harmonic half-plane mapping $L$ given by \eqref{eq1.5}.

\begin{example}\label{ex2.9}
Since
\[\frac{\partial}{\partial \theta}\left(\arg \left\{\frac{\partial}{\partial \theta}L(r e^{i \theta})\right\}\right)=\frac{1-6r^2+r^4+12r^2\cos^2{\theta}-4r(1+r^2)\cos^3 \theta}{1+(2 \cos^2 \theta -3)[4 \cos \theta (1+r^2)-6r]r+r^4},\]
therefore if we set
\[p(r,u)=1-6r^2+r^4+12r^2 u^2-4r(1+r^2)u^3-\alpha [1+(2u^2-3)\{4u(1+r^2)-6r\}r+r^4]\]
where $u=\cos \theta$, then $p(r,-1)>0$ and $p(r,1)>0$ for $r \in (0,1)$ and $\alpha \in [0,1)$. Also, further analysis shows that $p(r,u)$ has a local minimum at $u=u_{0}$, where $u_{0}$ is given by
\[u_{0}=\frac{r(1+\alpha)-\sqrt{\alpha(1+2\alpha)(1+r^4)+(1+4\alpha+5\alpha^2)r^2}}{(1+r^2)(1+2\alpha)}.\]
Consequently, it follows that $L$ is fully convex of order $\alpha$ in $|z|<r_{C}$, where $r_{C}=r_{C}(\alpha)$ is the positive root of the equation $p(r,u_{0})=0$. In particular, $r_{C}(0)=\sqrt{2}-1$,
\[r_{C}\left(\frac{1}{4}\right)=\frac{1}{3}\sqrt{\frac{1}{3}(223-70\sqrt{10})}\approx 0.246499, \quad r_{C}\left(\frac{1}{2}\right)=\frac{1}{\sqrt{26+15\sqrt{3}}}\approx0.138701\]
and
\[ r_{C}\left(\frac{3}{4}\right)=\sqrt{\frac{5}{681+182\sqrt{14}}}\approx0.0605898.\]

Regarding the fully starlikeness of $L$, note that
\[\frac{\partial}{\partial \theta}\arg L(r e^{i \theta})=\frac{(1-r^2)[1+(2\cos^2 \theta -5)r \cos\theta+3r^2- r^3 \cos\theta]}{(1 - 2 r \cos\theta+ r^2)^2 (\cos \theta - r)^2 + (1 - \cos^2\theta) (1 - r^2)^2}.\]
Considering the function
\[q(r,u)=(1-r^2)[1+u(2u^2-5)r+3r^2-ur^3]-\alpha[(1 - 2 r u + r^2)^2 (u - r)^2 + (1 - u^2) (1 - r^2)^2]\]
where $u=\cos \theta$, we see that for $\alpha \in (0,1)$, $q(r,u)\geq 0$ for $-1 \leq u \leq 1$ if and only if
\[q(r,-1)=(1+r)^4[1-r-\alpha(1+r)^2] \geq 0.\]
This inequality implies that $r \leq r_{S}$ where $r_{S}=r_{S}(\alpha)$ is given by
\[r_{S}(\alpha)=\frac{\sqrt{1+8\alpha}-(1+2\alpha)}{2\alpha}.\]
This shows that $L$ is fully starlike of order $\alpha$ in $|z|<r_{S}$, provided $\alpha \in (0,1)$. The case $\alpha=0$ is being discussed in Example \ref{ex1.1}. Note that $r_{C}(0)=r_{S}(1-1/\sqrt{2})$.
\end{example}

\begin{remark}
Given $f \in \mathcal{FK}_{H}$, it will be interesting to determine $\alpha \in [0,1)$ for which $f \in \mathcal{FS}^{*}_{H}(\alpha)$. For instance, the function $L(az)/a$, where $a=\sqrt{2}-1$ and $L$ is given by \eqref{eq1.5}, belongs to the class $\mathcal{FK}_{H}$ and Example \ref{ex2.9} shows that $L(az)/a\in \mathcal{FS}^{*}_{H}(1-1/\sqrt{2})$. This example motivates the following problem:

\textbf{Problem.} To determine $\alpha \in [0,1-1/\sqrt{2}]$ such that $\mathcal{FK}_{H}\subset\mathcal{FS}^{*}_{H}(\alpha)$.\\
This problem may be regarded as the harmonic analogue of well-known Marx-Strohh\"{a}cker inequality.
\end{remark}

The next theorem deals with the harmonic analogue of Alexander's theorem in context of fully convex and fully starlike mappings of order $\alpha$ $(0\leq \alpha<1)$. The proof being similar to \cite[Theorem 4, p. 140]{chuaqui} is omitted.
\begin{theorem}\label{th2.10}
Let $h$, $g$, $H$ and $G$ be analytic functions in the unit disc $\mathbb{D}$, related by
\[zH'(z)=h(z)\quad\mbox{and}\quad zG'(z)=-g(z)\]
Then, $f=h+\bar{g}$ is fully starlike of order $\alpha$ if and only if $F=H+\bar{G}$ is fully convex of order $\alpha$, where $0\leq \alpha<1$.
\end{theorem}

This theorem provides an abundant examples of fully convex and fully starlike mappings of order $\alpha$ $(0\leq \alpha<1)$. For instance, since the functions $f_{n}$ defined in Example \ref{ex2.5} are fully starlike in $\mathbb{D}$, the functions $F_{n}(z)=z-{(1-\alpha)/(n(n+\alpha))\bar{z}^n}$ are fully convex of order $\alpha$. Similarly, since the function $L(r_{S}z)$ is fully starlike of order $\alpha$, $r_{S}=r_{S}(\alpha)$ being the radius of fully starlikeness of order $\alpha$ for $L$ determined in Example \ref{ex2.9}, the function
\[p(z)=\RE \frac{1}{1-r_{S}z}-i \arg (1-r_{S}z)\]
is fully convex of order $\alpha$.

The next example shows that Theorem \ref{th2.10} does not have full generality to the classes $\mathcal{FK}_{H}(\alpha)$ and $\mathcal{FS}^{*}_{H}(\alpha)$. \begin{example}
If $r_{C}=r_{C}(\alpha)$ is the radius of fully convexity of order $\alpha$ for the mapping $L$ determined in Example \ref{ex2.9}, then the function $F(z)=L(r_{C}z)/r_{C}=H(z)+\overline{G(z)} \in \mathcal{FK}_{H}(\alpha)$ and the corresponding function $f(z)=h(z)+\overline{g(z)}$, where
\[h(z)=zH'(z)=\frac{z}{(1-r_{C}z)^{3}}\quad \mbox{and}\quad g(z)=-zG'(z)=\frac{r_{C}z^2}{(1-r_{C}z)^3}\]
is not even locally univalent, since the Jacobian of $f$ given by
\[J_{f}(z)=\frac{1-r_{C}^2|z|^{2}}{|1-r_{C}z|^{8}}(1+r_{C}^2|z|^2+4r_{C}\RE z)\]
vanishes at $z=(2-\sqrt{3})/r_{C}$. However, by Theorem \ref{th2.10}, $f$ is fully starlike of order $\alpha$. It is easily seen that $f \in \mathcal{FS}^{*}_{H}(\alpha)$ for $|z|<(2-\sqrt{3})/r_{C}$.
\end{example}

The partial analogue of Theorem \ref{th2.10} to the classes  $\mathcal{FK}_{H}(\alpha)$ and $\mathcal{FS}^{*}_{H}(\alpha)$ is given in the following corollary.
\begin{corollary}\label{cor2.13}
If $f=h+\bar{g} \in \mathcal{FS}^{*}_{H}(\alpha)$ $(0\leq \alpha<1)$ and if $H$ and $G$ are the analytic functions defined by
\[zH'(z)=h(z),\quad zG'(z)=-g(z), \quad \mbox{and} \quad H(0)=G(0)=0\]
then $F=H+\bar{G} \in \mathcal{FK}_{H}(\alpha)$.
\end{corollary}

\begin{proof}
By Theorem \ref{th2.10}, $F$ is fully convex of order $\alpha$ and hence univalent. By hypothesis, the Jacobian $J_{F}(z) \neq 0$ for each $z \in \mathbb{D}$ and since $J_{F}(0)=|H'(0)|^{2}-|G'(0)|^{2}=1-|g'(0)|^{2}>0$, it follows that $F$ is sense-preserving in $\mathbb{D}$ so that $F \in \mathcal{FK}_{H}(\alpha)$.
\end{proof}

Note that the exact radius of fully starlikeness of order $\alpha$ $(0\leq \alpha <1)$ for the subclasses $S^{*}_{H}$, $K_{H}$ and $C_{H}$ in $S_{H}$ is still unknown. The results in this direction are being investigated in the next section. However, if $\alpha=0$, we have the following result.

\begin{theorem}\label{th2.14}
Suppose that $f=h+\bar{g}\in \mathcal{S}_{H}$.
\begin{itemize}
  \item [$(i)$] If $f \in \mathcal{K}_{H}$ then $f$ is fully starlike in at least $|z|<4\sqrt{2}-5$;
  \item [$(ii)$] If $f \in \mathcal{C}_{H}$ then $f$ is fully starlike in at least $|z|<3-\sqrt{8}$;
  \item [$(iii)$] If $f \in \mathcal{S}_{H}^{*}$ then $f$ is fully starlike in at least $|z|<\sqrt{2}-1$.
\end{itemize}
\end{theorem}

\begin{proof}
Firstly, we will prove $(i)$. Since $f \in \mathcal{K}_{H}$, the analytic functions $h+\epsilon g$ are close-to-convex for each $|\epsilon|=1$ by \cite[Theorem 5.7, p.15]{cluniesheilsmall}. Since the radius of starlikeness in close-to-convex analytic mappings is $4\sqrt{2}-5$, the functions $h+\varepsilon g$ are starlike in $|z|<4\sqrt{2}-5$. By Theorem \ref{starlike}, $f$ is fully starlike in $|z|<4\sqrt{2}-5$. This proves $(i)$. The proof of part $(ii)$ follows from the fact that $f \in \mathcal{C}_{H}$ is fully convex (and hence fully starlike) in $|z|<3-\sqrt{8}$. Regarding the proof of $(iii)$, if $f \in \mathcal{S}^{*}_{H}$ then by \cite[Lemma, p. 108]{duren}, the function $F=H+\overline{G}\in \mathcal{K}_{H}$, where $zH'(z)=h(z)$, $zG'(z)=-g(z)$, and $H(0)=G(0)=0$ so that $f$ is fully starlike in at least $|z|<\sqrt{2}-1$.
\end{proof}


\section{Radii Problems}\label{sec3}
In this section, we generalize the results given in \cite{ponnusamy} for fully starlike functions of order $\alpha$ and fully convex functions of order $\alpha$. The proof of the theorems follow from an easy modification of the proof of the corresponding results from \cite{ponnusamy}. For the sake of completeness, we include the details.  The following identities are quite useful in the proof of the theorems:
\begin{equation}\label{eq3.1}
\begin{split}
\frac{r}{(1-r)^2}&=\sum_{n=1}^{\infty}nr^{n}\quad\quad,\quad\quad \frac{r(1+r)}{(1-r)^3}=\sum_{n=1}^{\infty}n^{2}r^{n},\\
\frac{r(r^2+4r+1)}{(1-r)^4}&=\sum_{n=1}^{\infty}n^{3}r^{n},\quad\mbox{and}\quad \frac{r(1+r)(1+10r+r^{2})}{(1-r)^5}=\sum_{n=1}^{\infty}n^{4}r^{n}.
\end{split}
\end{equation}

\begin{theorem}\label{th3.1}
Let $h$ and $g$ have the form \eqref{eq1.1}, $0\leq \alpha<1$ and the coefficients of the series satisfy the conditions \eqref{eq1.2}. Then $f=h+\bar{g}$ is univalent and fully starlike of order $\alpha$ in the disk $|z|<r_{S}$, where $r_{S}=r_{S}(\alpha)$ is the real root of the equation
\begin{equation}\label{eq3.2}
2(1-\alpha)(1-r)^4+\alpha(1-r)^2-(1+r)^2=0
\end{equation}
in the interval $(0,1)$. Moreover, this result is sharp for each $\alpha \in [0,1)$.
\end{theorem}

\begin{proof}
The coefficient conditions \eqref{eq1.2} imply that $b_{1}=0$ and $h$ and $g$ are analytic in $\mathbb{D}$. Thus, $f=h+\bar{g}$ is harmonic in $\mathbb{D}$. Let $0<r<1$. It suffices to show that $f_{r} \in \mathcal{FS}^{*}_{H}(\alpha)$, where $f_{r}$ is defined by
\begin{equation}\label{eq3.3}
f_{r}(z)=\frac{f(rz)}{r}=z+\sum_{n=2}^{\infty}a_{n}r^{n-1}z^n+\sum_{n=2}^{\infty}b_{n}r^{n-1}z^n,\quad z\in \mathbb{D}.
\end{equation}
Consider the sum
\begin{equation}\label{eq3.4}
S=\sum_{n=2}^{\infty}\frac{n-\alpha}{1-\alpha}|a_{n}|r^{n-1}+\sum_{n=2}^{\infty}\frac{n+\alpha}{1-\alpha}|b_{n}|r^{n-1}.
\end{equation}
Using the coefficient bounds \eqref{eq1.2} and simplifying, we have
\[S \leq \frac{1}{3(1-\alpha)}\left[2\sum_{n=2}^{\infty} n^3 r^{n-1}+(1-3\alpha)\sum_{n=2}^{\infty}nr^{n-1}\right].\]
According to Lemma \ref{lem2.6}, we need to show that $S \leq 1$ or equivalently $r$ satisfies the inequality
\[2\sum_{n=2}^{\infty} n^3 r^{n-1}+(1-3\alpha)\sum_{n=2}^{\infty}nr^{n-1} \leq 3(1-\alpha).\]
Using the identities \eqref{eq3.1}, the last inequality reduces to
\[\frac{(1+r)^2}{(1-r)^4}-\frac{\alpha}{(1-r)^2}\leq 2(1-\alpha),\]
or
\[2(1-\alpha)(1-r)^4+\alpha(1-r)^2-(1+r)^2\geq 0.\]
Thus, by Lemma \ref{lem2.6}, $f_{r} \in \mathcal{FS}^{*}_{H}(\alpha)$ for $r\leq r_{S}$ where $r_{S}$ is the real root of \eqref{eq3.2} in $(0,1)$. In particular, $f$ is univalent and fully starlike of order $\alpha$ in $|z|<r_{S}$.

To prove the sharpness, consider the function $f_{0}(z)=h_{0}(z)+\overline{g_{0}(z)}$, where
\[h_{0}(z)=2z-\frac{z-\frac{1}{2}z^2+\frac{1}{6}z^3}{(1-z)^3}\quad \mbox{and}\quad g_{0}(z)=\frac{\frac{1}{2}z^2+\frac{1}{6}z^3}{(1-z)^3},\]
so that
\[f_{0}(z)=z-\frac{1}{6}\sum_{n=2}^{\infty}(n+1)(2n+1)z^{n}+\overline{\frac{1}{6}\sum_{n=2}^{\infty}(n-1)(2n-1)z^{n}}.\]
As $f_{0}$ has real coefficients, we obtain for $r \in (0,1)$
\begin{align*}
    J_{f_{0}}(r)&=(h_{0}'(r)+g_{0}'(r))(h_{0}'(r)-g_{0}'(r))\\
                &=\frac{(1-7r+6r^2-2r^3)(1-10r+11r^2-8r^3+2r^4)}{(1-r)^7}.
\end{align*}
Note that the roots of the equation \eqref{eq3.2} in $(0,1)$ are decreasing as a function of $\alpha \in [0,1)$. Consequently, $r_{S}(\alpha)\leq r_{S}(0)\approx0.112903$ and as $J_{f_{0}}(r_{S}(0))=0$, therefore in view of Lewy's theorem, the function $f_{0}$ is not univalent in $|z|<r$ if $r>r_{S}(0)$. Also, since
\[\left.\frac{\partial}{\partial\theta}\arg f_{0}(r e^{i\theta})\right|_{\theta=0}=\frac{rh_{0}'(r)-rg_{0}'(r)}{h_{0}(r)+g_{0}(r)}=\frac{1-10r+11r^2-8r^3+2r^4}{(1-r)^2(2r^2-4r+1)},\]
therefore if $z=r_{S}$, where $r_{S}$ is the real root of \eqref{eq3.2} in $(0,1)$, then
\[\left.\frac{\partial}{\partial\theta}\arg f_{0}(r e^{i\theta})\right|_{\theta=0,r=r_{S}}=\alpha,\]
showing that the bound $r_{S}$ is best possible.
\end{proof}

If $\alpha=0$ then Theorem \ref{th3.1} simplifies to \cite[Theorem 1.5]{ponnusamy}. Also, Theorem \ref{th3.1} readily gives the following corollary.
\begin{corollary}\label{cor3.2}
Let $f \in \mathcal{S}^{*0}_{H}$ (resp. $\mathcal{C}^{0}_{H}$) and $0 \leq \alpha <1$. Then $f$ is fully starlike of order $\alpha$ in at least $|z|<r_{S}$, where $r_{S}$ is the real root of \eqref{eq3.2} in $(0,1)$.
\end{corollary}

Proceeding in a similar manner as in Theorem \ref{th3.1} and invoking Lemma \ref{lem2.2} instead of Lemma \ref{lem2.6}, we have the following result.
\begin{theorem}\label{th3.3}
Under the hypothesis of Theorem \ref{th3.1}, $f=h+\overline{g}$ is univalent and fully convex of order $\alpha$ in the disk $|z|<r_{C}$, where $r_{C}=r_{C}(\alpha)$ is the real root of the equation
\begin{equation}\label{eq3.5}
2(1-\alpha)(1-r)^5+\alpha (1+r)(1-r)^2-(1+r)(r^2+6r+1)=0
\end{equation}
in the interval $(0,1)$. In particular, $f$ is univalent and fully convex in $|z|<r_{C}(0)\approx 0.0614313$.
\end{theorem}

The bound $r_{C}$ given by \eqref{eq3.5} is sharp by considering the function $f_{0}(z)=2z-K(z)$ where $K$ is given by \eqref{eq1.3}. In fact, as $f_{0}$ has real coefficients, we obtain
\[\left.\frac{\partial}{\partial \theta}\left(\arg \left\{\frac{\partial}{\partial \theta}f_{0}(r e^{i \theta})\right\}\right)\right|_{\theta=0,r=r_{C}}=\frac{1-17 r_{C}+13r^2_{C}-21 r_{C}^{3}+10r_{C}^{4}-2r_{C}^{5}}{(1-r_{c})^{2}(1-7r_{C}+6r_{C}^{2}-2r_{C}^{3})}=\alpha.\]
Theorem \ref{th3.3} immediately gives

\begin{corollary}\label{cor3.4}
Let $f \in \mathcal{S}^{*0}_{H}$ (resp. $\mathcal{C}^{0}_{H}$) and $0 \leq \alpha <1$. Then $f$ is fully convex of order $\alpha$ in at least $|z|<r_{C}$, where $r_{C}$ is the real root of \eqref{eq3.5} in $(0,1)$.
\end{corollary}

It is clear that the result in Corollary \ref{cor3.4} is not sharp if $\alpha=0$. Corresponding to Theorem \ref{th3.1}, the next theorem determines the radius of univalence and fully starlikeness of order $\alpha$ for functions $f=h+\bar{g} \in \mathcal{H}$, where the Taylor coefficients of the series of $h$ and $g$ satisfy \eqref{eq1.4}.

\begin{theorem}\label{th3.5}
Let $h$ and $g$ have the form \eqref{eq1.1}, $0\leq \alpha<1$ and the coefficients of the series satisfy the conditions \eqref{eq1.4}. Then $f=h+\bar{g}$ is univalent and fully starlike of order $\alpha$ in the disk $|z|<r_{S}$, where $r_{S}=r_{S}(\alpha)$ is the real root of the equation
\begin{equation}\label{eq3.6}
(2-\alpha)(1-r)^3+\alpha  r(1-r)^2-1-r=0
\end{equation}
in the interval $(0,1)$. Moreover, this result is sharp for each $\alpha \in [0,1)$.
\end{theorem}

\begin{proof}
Following the notation and the method of the proof of Theorem \ref{th3.1}, it suffices to show that $f_{r} \in \mathcal{FS}^{*}_{H}(\alpha)$. Considering the sum \eqref{eq3.4} and using the coefficient bounds \eqref{eq1.4}, we have on simplification
\[S \leq \frac{1}{1-\alpha}\left[\sum_{n=2}^{\infty}n^2r^{n-1}-\alpha \sum_{n=2}^{\infty} r^{n-1}\right].\]
By Lemma \ref{lem2.6}, we need to show that $S \leq 1$ or equivalently
\[\sum_{n=2}^{\infty}n^2r^{n-1}-\alpha \sum_{n=2}^{\infty} r^{n-1} \leq 1-\alpha.\]
Using the identities \eqref{eq3.1}, the last inequality reduces to
\[\frac{1+r}{(1-r)^3}-\frac{\alpha r}{1-r}\leq 2-\alpha\]
or
\[(2-\alpha)(1-r)^3+\alpha  r(1-r)^2-1-r \geq 0.\]
Thus, by Lemma \ref{lem2.6}, we deduce that $f$ is univalent and fully starlike of order $\alpha$ in $|z|<r_{S}$, where $r_{S}$ is the real root of \eqref{eq3.6}.

The sharpness part of the theorem follows if we consider the function $f_{0}=2z-L(z)$, where $L$ is given by \eqref{eq1.5} so that
\[f_{0}(z)=z-\sum_{n=2}^{\infty}\frac{n+1}{2}z^{n}+\overline{\sum_{n=2}^{\infty}\frac{n-1}{2}z^{n}}.\]
As $f_{0}$ has real coefficients, we obtain for $r \in (0,1)$
\[J_{f_{0}}(r)=\frac{(1-4r+r^2)(1-7r+6r^2-2r^3)}{(1-r)^5}\]
Again observe that the roots of the equation \eqref{eq3.6} in $(0,1)$ are decreasing as a function of $\alpha \in [0,1)$. Consequently, $r_{S}(\alpha)\leq r_{S}(0)\approx0.16487$ and as $J_{f_{0}}(r_{S}(0))=0$, therefore by Lewy's theorem, we deduce that the function $f_{0}$ is not univalent in $|z|<r$ if $r>r_{S}(0)$. Also,
\[\left.\frac{\partial}{\partial\theta}\arg f_{0}(r e^{i\theta})\right|_{\theta=0,r=r_{S}}=\frac{1-7r_{S}+6r_{S}^2-2r_{S}^3}{(1-r_{S})^2(1-2r_{S})}=\alpha,\]
showing that the bound $r_{S}$ is best possible.
\end{proof}

Note that Theorem \ref{th3.5} reduces to \cite[Theorem 1.9]{ponnusamy} in case $\alpha=0$. Moreover, Theorem \ref{th3.5} quickly yields
\begin{corollary}\label{cor3.6}
Let $f \in \mathcal{K}_{H}^{0}$ and $0\leq \alpha<1$. Then $f$ is fully starlike of order $\alpha$ in at least $|z|<r_{S}$, where $r_{S}$ is the real root of \eqref{eq3.6} in $(0,1)$.
\end{corollary}

It is expected that Corollary \ref{cor3.6} can be further improved and since the harmonic half-plane mapping $L$ given by \eqref{eq1.5} is extremal in $\mathcal{K}_{H}^{0}$, therefore Examples \ref{ex1.1} and \ref{ex2.9} motivates the following conjecture:

\textbf{Conjecture A.} If $f \in \mathcal{K}_{H}^{0}$, then $f$ is fully starlike of order $\alpha$ ($0\leq \alpha<1$) in $|z|<r_{S}$ where $r_{S}=r_{S}(\alpha)$ is given by
\[r_{S}(\alpha)=\left\{
                  \begin{array}{ll}
                    \dfrac{\sqrt{1+8\alpha}-(1+2\alpha)}{2\alpha}, & \hbox{if $\alpha \in(0,1)$;} \\
                    \sqrt{\dfrac{7\sqrt{7}-17}{2}}, & \hbox{if $\alpha=0$.}
                  \end{array}
                \right.\]

Theorem \ref{th2.14} shows that the results in Corollaries \ref{cor3.2} and \ref{cor3.6} are not sharp if $\alpha=0$. Using Lemma \ref{lem2.2} and proceeding in a similar manner as in Theorem \ref{th3.5}, we obtain the following result.

\begin{theorem}\label{th3.7}
Under the hypothesis of Theorem \ref{th3.5}, $f=h+\bar{g}$ is univalent and fully convex of order $\alpha$ in the disk $|z|<r_{C}$, where $r_{C}=r_{C}(\alpha)$ is the real root of the equation
\begin{equation}\label{eq3.7}
2(1-\alpha)(1-r)^4+\alpha(1-r)^2-(r^2+4r+1)=0
\end{equation}
in the interval $(0,1)$. In particular, $f$ is univalent and fully convex in $|z|<r_{C}(0)\approx0.0903331$.
\end{theorem}

The radius bound $r_{C}$ given by \eqref{eq3.7} is sharp for each $\alpha \in [0,1)$ by considering the function $f_{0}(z)=h_{0}(z)+\overline{g_{0}(z)}$, where
\[h_{0}(z)=2z-\frac{1}{2}\left(\frac{z}{1-z}+\frac{z}{(1-z)^2}\right)\quad\mbox{and}\quad g_{0}(z)=\frac{1}{2}\left(\frac{z}{1-z}-\frac{z}{(1-z)^2}\right)\]
and noticing that
\[\left.\frac{\partial}{\partial \theta}\left(\arg \left\{\frac{\partial}{\partial \theta}f_{0}(r e^{i \theta})\right\}\right)\right|_{\theta=0,r=r_{C}}=\frac{1-12r_{C}+11r_{C}^2-8r_{C}^3+2r_{C}^4}{(1-r_{C})^2(1-4r_{C}+2r_{C}^2)}=\alpha.\]

An immediate corollary to Theorem \ref{th3.7} states that
\begin{corollary}\label{cor3.8}
If $f \in \mathcal{K}_{H}^{0}$ and $0\leq \alpha<1$, then $f$ is fully convex of order $\alpha$ in $|z|<r_{C}$, where $r_{C}$ is the real root of \eqref{eq3.7}.
\end{corollary}

It is known that the result given in Corollary \ref{cor3.8} is not sharp if $\alpha=0$. Since the harmonic half-plane mapping $L$ given by \eqref{eq1.5} gives the sharp bound for $\alpha=0$, therefore Example \ref{ex2.9} motivates the following conjecture:

\textbf{Conjecture B.} If $f \in \mathcal{K}_{H}^{0}$, then $f$ is fully convex of order $\alpha$ ($0\leq \alpha<1$) in $|z|<r_{S}$ where $r_{C}=r_{C}(\alpha)$ is the positive root of the equation $p(r,u_{0})=0$ in $(0,1)$ with
\[p(r,u)=1-6r^2+r^4+12r^2 u^2-4r(1+r^2)u^3-\alpha [1+(2u^2-3)\{4u(1+r^2)-6r\}r+r^4],\]
and
\[u_{0}=\frac{r(1+\alpha)-\sqrt{\alpha(1+2\alpha)(1+r^4)+(1+4\alpha+5\alpha^2)r^2}}{(1+r^2)(1+2\alpha)}.\]
For $\alpha=0$, this conjecture has been solved (see \cite{ruscheweyh}).


\section{Harmonic convolution}\label{sec4}
Consider the half-plane mapping $L$ in $\mathcal{K}_{H}^{0}\subset \mathcal{K}_{H}$, given by \eqref{eq1.5}. The coefficients of the product $L*L$ are too large for this product to be in $\mathcal{K}_{H}$. In fact, the image of the unit disk $\mathbb{D}$ under $L*L$ is $\mathbb{C}\backslash(-\infty,-1/4]$, which is not a convex domain. However, $L*L \in \mathcal{S}_{H}^{*0}$, by \cite[Theorem 3]{dorff2}. Consider the following example.

\begin{example}\label{ex4.1}
We show that $L*L$ maps the subdisks $|z|<r$ onto convex domain precisely for $r \leq 2-\sqrt{3}$. Since we can write $L$ as
\[L(z)=\frac{1}{2}\frac{z}{1-z}+\frac{1}{2}\frac{z}{(1-z)^2}+\overline{\frac{1}{2}\frac{z}{1-z}-\frac{1}{2}\frac{z}{(1-z)^2}},\]
and using the fact that for an analytic function $\varphi$ with $\varphi(0)=0$, we have
\[\frac{z}{1-z}*\varphi(z)=\varphi(z)\quad\mbox{and}\quad \frac{z}{(1-z)^2}*\varphi(z)=z\varphi'(z),\]
it follows that
\[(L*L)(z)=\frac{1}{4}\frac{z}{1-z}+\frac{1}{2}\frac{z}{(1-z)^2}+\frac{1}{4}\frac{z(1+z)}{(1-z)^3}+
\overline{\frac{1}{4}\frac{z}{1-z}-\frac{1}{2}\frac{z}{(1-z)^2}+\frac{1}{4}\frac{z(1+z)}{(1-z)^3}}.\]

\begin{figure}[hb]
\centering
\includegraphics[width=2.5in]{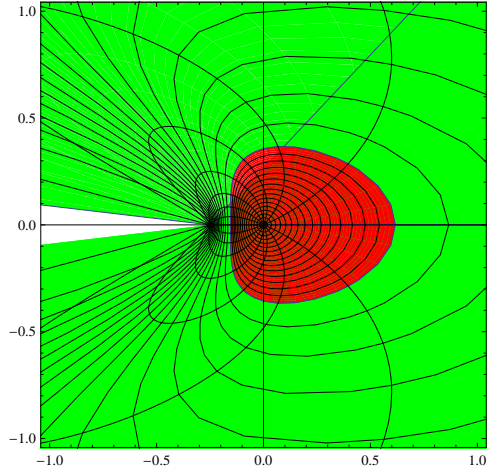}
\caption{Image of the subdisk $|z|<2-\sqrt{3}$ under the mapping $L*L$.}\label{fig2}
\end{figure}

After simplification, we get
\[(L*L)(z)=\frac{1}{2}\RE \frac{z(z^{2}-z+2)}{(1-z)^{3}}+i\IM \frac{z}{(1-z)^{2}}.\]
To prove our assertion, it will be necessary to study the change of the tangent direction
\[\Psi_{r}(\theta)=\arg\left\{\frac{\partial}{\partial\theta}(L*L)(re^{i\theta})\right\}\]
of the image curve as the point $z= r e^{i\theta}$ moves around the circle $|z|=r$. Note that
\[\frac{\partial}{\partial\theta}(L*L)(re^{i\theta})=A(r,\theta)+iB(r,\theta),\]
where
\[|1-z|^{8}A(r,\theta)=-r(1-r^{2})[r^{2}\sin{3\theta}+r(1+r^{2})\sin{2\theta}+(1-9r^{2}+r^{4})\sin{\theta}]\]
and
\[|1-z|^{6}B(r,\theta)=r(1-r^{2})[(1+r^{2})\cos{\theta}-2r (1+\sin^{2}{\theta})],\]
so that
\[\tan \Psi_{r}(\theta)=\frac{B(r,\theta)}{A(r,\theta)}=
\frac{(1-2r\cos{\theta}+r^{2})[2r(1+\sin^{2}{\theta})-(1+r^{2})\cos{\theta}]}{r^{2}\sin{3\theta}+r(1+r^{2})\sin{2\theta}+(1-9r^{2}+r^{4})\sin{\theta}}.\]
A lengthy calculation leads to an expression for the derivative in the form
\[[4r^2 u^2+2r(1+r^2)u+(1-10r^2+r^4)]^{2}\frac{\partial}{\partial\theta}\tan \Psi_{r}(\theta)=q(r,u),\]
where $u=\cos \theta$ and
\begin{align*}
q(r,u)=1+2u(u^2-2)r&+8(1-4u^2+2u^4)r^2+2u(34-21u^2+4u^4)r^3-2(41-24u^2+8u^4)r^4\\
&+2u(34-21u^2+4u^4)r^5+8(1-4u^2+2u^4)r^6+2u(u^2-2)r^7+r^8.
\end{align*}
Observe that the roots of $q(r,u)=0$ in $(0,1)$ are increasing as a function of $u \in [-1,1]$. Consequently, it follows that $q(r,u)\geq 0$ for $-1\leq u\leq 1$ if and only if
\[q(r,-1)=1+2r-8r^2-34r^3-50r^4-34r^5-8r^6+2r^7+r^8\geq 0.\]
This inequality implies that $r\leq 2-\sqrt{3}$. This proves that the tangent angle $\Psi_{r}(\theta)$ increases monotonically with $\theta$ if $r\leq 2-\sqrt{3}$ but is not monotonic for $2-\sqrt{3}<r<1$. Thus, the harmonic mapping $L*L$ sends each disk $|z|<r\leq 2-\sqrt{3}$ to a convex region, but the image is not convex when $2-\sqrt{3}<r<1$ (see Figure \ref{fig2}).
\end{example}

\begin{theorem}\label{th4.2}
Let $h$ and $g$ have the form \eqref{eq1.1}, $0\leq \alpha<1$ and the coefficients of the series satisfy the conditions
\[|a_{n}|\leq\left(\frac{n+1}{2}\right)^{2}\quad\mbox{and}\quad|b_{n}|\leq\left(\frac{n-1}{2}\right)^2,\quad \mbox{for all }n \geq 1.\]
Then $f=h+\bar{g}$ is univalent and fully starlike of order $\alpha$ in the disk $|z|<r_{0}$, where $r_{0}=r_{0}(\alpha)$ is the real root of the equation
\begin{equation}\label{eq4.1}
2(1-\alpha)(1-r)^4+\alpha(1-r)^2-(r^2+r+1)=0
\end{equation}
in the interval $(0,1)$. In particular, $f$ is univalent and fully starlike in $|z|<r_{0}$, where $r_{0}=r_{0}(0)\approx0.129831$ is the root of the biquadratic equation $2r^4-8r^3+11r^2-9r+1=0$. Moreover, the result is sharp for each $\alpha \in [0,1)$.
\end{theorem}

\begin{proof}
Following the method of the proof of Theorem \ref{th3.1}, it suffices to show that $f_{r} \in \mathcal{FS}^{*}_{H}(\alpha)$, where $f_{r}$ is defined by \eqref{eq3.3}. Considering the sum \eqref{eq3.4} and using the coefficient bounds, we have
\[S \leq \frac{1}{2(1-\alpha)}\left[\sum_{n=2}^{\infty}n^3r^{n-1}+(1-2\alpha)\sum_{n=2}^{\infty}nr^{n-1}\right].\]
According to Lemma \ref{lem2.6}, we need to show that $S\leq1$ or equivalently $r$ satisfies the inequality
\[\sum_{n=2}^{\infty}n^3r^{n-1}+(1-2\alpha)\sum_{n=2}^{\infty}nr^{n-1}\leq 2(1-\alpha).\]
Using the identities \eqref{eq3.1}, the last inequality reduces to
\[\frac{r^2+r+1}{(1-r)^4}-\frac{\alpha}{(1-r)^2}\leq 2(1-\alpha),\]
or
\[2(1-\alpha)(1-r)^4+\alpha(1-r)^2-(r^2+r+1)\geq 0.\]
Thus, by Lemma \ref{lem2.6}, $f_{r} \in \mathcal{FS}^{*}_{H}(\alpha)$ for $r\leq r_{0}$ where $r_{0}=r_{0}(\alpha)$ is the real root of the equation \eqref{eq4.1} in $(0,1)$. In particular, $f$ is univalent and fully starlike of order $\alpha$ in $|z|<r_{0}$.

Next, to prove the sharpness of the statement of the theorem, we consider the function
\[f_{0}(z)=h_{0}(z)+\overline{g_{0}(z)},\]
with
\[h_{0}(z)=2z-\frac{1}{4}\frac{z}{1-z}-\frac{1}{2}\frac{z}{(1-z)^2}-\frac{1}{4}\frac{z(1+z)}{(1-z)^3}\quad \mbox{and}\quad g_{0}(z)=\frac{1}{4}\frac{z}{1-z}-\frac{1}{2}\frac{z}{(1-z)^2}+\frac{1}{4}\frac{z(1+z)}{(1-z)^3}.\]
We note that
\[f_{0}(z)=z-\sum_{n=2}^{\infty}\left(\frac{n+1}{2}\right)^2z^{n}+\overline{\sum_{n=2}^{\infty}\left(\frac{n-1}{2}\right)^2z^{n}}.\]
As $f_{0}$ has real coefficients, we obtain that for $r \in (0,1)$
\begin{align*}
    J_{f_{0}}(r)&=(h_{0}'(r)+g_{0}'(r))(h_{0}'(r)-g_{0}'(r))\\
                &=\left(2-\frac{1+r}{(1-r)^3}\right)\left(2-\frac{1}{2}\frac{1}{(1-r)^2}-\frac{1}{2}\frac{1+4r+r^2}{(1-r)^4}\right)\\
                &=\frac{(1-7r+6r^2-2r^3)_(1-9r+11r^2-8r^3+2r^4)}{(1-r)^7}.
\end{align*}
Note that the roots of the equation \eqref{eq4.1} in $(0,1)$ are decreasing as a function of $\alpha \in [0,1)$. Consequently, $r_{0}(\alpha)\leq r_{0}(0)\approx0.112903$ and as $J_{f_{0}}(r_{0}(0))=0$, therefore in view of Lewy's theorem, the function $f_{0}$ is not univalent in $|z|<r$ if $r>r_{0}(0)$. Also, regarding starlikeness, we observe that
\[\left.\frac{\partial}{\partial\theta}\arg f_{0}(r e^{i\theta})\right|_{\theta=0}=\frac{rh_{0}'(r)-rg_{0}'(r)}{h_{0}(r)+g_{0}(r)}=\frac{1-9r+11r^2-8r^3+2r^4}{(1-r)^2(2r^2-4r+1)},\]
therefore if we set $z=r_{0}$, where $r_{0}$ is the real root of \eqref{eq4.1} in $(0,1)$, then
\[\left.\frac{\partial}{\partial\theta}\arg f_{0}(r e^{i\theta})\right|_{\theta=0,r=r_{0}}=\alpha,\]
showing that the bound $r_{0}$ is best possible.
\end{proof}

An immediate consequence of Theorem \ref{th4.2} gives the following result.
\begin{corollary}\label{cor4.3}
Let $f,g \in \mathcal{K}_{H}^{0}$ and $0\leq \alpha<1$. Then $f*g$ is univalent and fully starlike of order $\alpha$ in  at least $|z|<r_{0}$ where $r_{0}=r_{0}(\alpha)$ is the real root of \eqref{eq4.1} in $(0,1)$. In particular $f*g$ is univalent and fully starlike in $|z|<r_{0}(0)\approx0.129831$.
\end{corollary}

Invoking Lemma \ref{lem2.2} instead of Lemma \ref{lem2.6} and proceeding in a similar manner as in Theorem \ref{th4.2}, we obtain the following result.
\begin{theorem}\label{th4.4}
Under the hypothesis of Theorem \ref{th4.2}, $f=h+\bar{g}$ is univalent and fully convex of order $\alpha$ in the disk $|z|<s_{0}$, where $s_{0}=s_{0}(\alpha)$ is the real root of the equation
\begin{equation}\label{eq4.2}
2(1-\alpha)(1-r)^5+\alpha(1+r)(1-r)^2-(1+r)(r^2+4r+1)=0
\end{equation}
in the interval $(0,1)$. In particular, $f$ is univalent and fully convex in $|z|<s_{0}(0)\approx0.0712543$.
\end{theorem}

It's worth to remark that the result regarding the univalence of $f$ in Theorem \ref{th4.4} can be further improved to $0.129831$ as seen by Theorem \ref{th4.2}. However, the estimate $s_{0}$ given by \eqref{eq4.2} regarding fully convexity of order $\alpha$ is sharp by considering the function $f_{0}(z)=2z-(L*L)(z)$, where $L$ is given by \eqref{eq1.5}. In fact, as $f_{0}$ has real coefficients, we obtain
\[\left.\frac{\partial}{\partial \theta}\left(\arg \left\{\frac{\partial}{\partial \theta}f_{0}(r e^{i \theta})\right\}\right)\right|_{\theta=0,r=s_{0}}=\frac{1-15s_{0}+15s_{0}^2-21s_{0}^3+10s_{0}^4-2s_{0}^5}{(1-s_{0})^2(1-7s_{0}+6s_{0}^2-2s_{0}^3)}=\alpha.\]
Theorem \ref{th4.4} easily gives
\begin{corollary}\label{cor4.5}
Let $f,g \in \mathcal{K}_{H}^{0}$ and $0\leq \alpha<1$. Then $f*g$ is univalent and fully convex of order $\alpha$ in at least $|z|<s_{0}$ where $s_{0}=s_{0}(\alpha)$ is the real root of \eqref{eq4.2} in $(0,1)$. In particular $f*g$ is univalent and fully convex in $|z|<s_{0}=s_{0}(0)\approx0.0712543$.
\end{corollary}
It is expected that Corollary \ref{cor4.5} can be further improved, and since the function $L$ given by \eqref{eq1.5} is extremal in $\mathcal{K}_{H}^{0}$ therefore in view of Example \ref{ex4.1} we have the following conjecture:

\textbf{Conjecture C.} If $f, g \in \mathcal{K}_{H}^{0}$, then $f*g$ is univalent and fully convex in $|z|<2-\sqrt{3}$.

Another similar result regarding convolution of analytic functions is that if $f_{1}\in \mathcal{K}$ and $f_{2}\in \mathcal{S}^{*}$, then $f_{1}*f_{2} \in \mathcal{S}^{*}$. To see this, observe that the functions $f_{n}=z+\bar{z}^{n}/n$ $(n\geq 2)$ defined in Example \ref{ex2.5} (with $\alpha=0$) are in $\mathcal{S}_{H}^{*0}$ and
\[(L*f_{n})(z)=z-\frac{n-1}{2n}\bar{z}^{n},\]
where $L$ is defined by \eqref{eq1.5}. Note that $L*f_{n}\in \mathcal{S}_{H}^{*0}$ if and only if $n=2,3$. Indeed, for $|z|=1$, we have
\[
\frac{\partial}{\partial \theta}\arg (L*f_{n})(z)=\RE\frac{z+\frac{n-1}{2}\bar{z}^{n}}{z-\frac{n-1}{2n}\bar{z}^{n}}\geq\frac{n(3-n)}{3n-1}.\]
This observation, together with the fact that $L*f_{n}$ is univalent only if $n=2,3$, it follows that $L*f_{n} \in \mathcal{S}_{H}^{*0}$ for $n=2,3$.

The next theorem follows from an easy modification of the proof of Theorem \ref{th4.2}.
\begin{theorem}\label{th4.6}
Let $h$ and $g$ have the form \eqref{eq1.1}, $0\leq \alpha<1$ and the coefficients of the series satisfy the conditions
\[|a_{n}|\leq\frac{1}{12}(n+1)^2(2n+1)\quad\mbox{and}\quad|b_{n}|\leq\frac{1}{12}(n-1)^2(2n-1),\quad \mbox{for all }n \geq 1.\]
Then $f=h+\bar{g}$ is univalent and fully starlike of order $\alpha$ in the disk $|z|<r_{0}$, where $r_{0}=r_{0}(\alpha)$ is the real root of the equation
\begin{equation}\label{eq4.3}
12(1-\alpha)(1-r)^5+\alpha(r^2+3r+6)(1-r)^2-6(1+r)^3=0
\end{equation}
in the interval $(0,1)$. In particular, $f$ is univalent and fully starlike in $|z|<r_{0}$, where $r_{0}=r_{0}(0)\approx0.0855165$ is the root of the equation $2r^5-10r^4+21r^3-17r^2+13r-1=0$. Moreover, the result is sharp.
\end{theorem}

To prove the sharpness of the result in Theorem \ref{th4.6}, we consider the function
\[f_{0}(z)=2z-(L*K)(z)=h_{0}(z)+\overline{g_{0}(z)},\]
where
\[h_{0}(z)=2z-\left(\frac{1}{4}\frac{z+\frac{1}{3}z^3}{(1-z)^3}+\frac{1}{4}\frac{z}{(1-z)^2}+\frac{1}{4}\frac{z(1+z)^2}{(1-z)^4}+\frac{1}{4}\frac{z(1+z)}{(1-z)^3}\right)\]
and
\[g_{0}(z)=-\left(\frac{1}{4}\frac{z+\frac{1}{3}z^3}{(1-z)^3}-\frac{1}{4}\frac{z}{(1-z)^2}-\frac{1}{4}\frac{z(1+z)^2}{(1-z)^4}+\frac{1}{4}\frac{z(1+z)}{(1-z)^3}\right).\]
Note that
\[f_{0}(z)=z-\sum_{n=2}^{\infty}\frac{1}{6}(n+1)^2(2n+1)z^n+\overline{\sum_{n=2}^{\infty}\frac{1}{6}(n-1)^2(2n-1)z^n}\]
and for $r \in(0,1)$, the Jacobian of $f_{0}$ is given by
\[J_{f_{0}}(r)=\frac{(1-13r+17r^2-21r^3+10r^4-2r^5)(1-11r+11r^2-8r^3+2r^4)}{(1-r)^9}\]
so that $J_{f_{0}}(r_{0}(0))=0$. Regarding starlikeness, observe that
\[\left.\frac{\partial}{\partial\theta}\arg f_{0}(r e^{i\theta})\right|_{\theta=0,r=r_{0}}=\frac{6(1-13r_{0}+17r^2-21r_{0}^3+10r_{0}^4-2r_{0}^5)}{(1-r_{0})^2(6-39r_{0}+35r_{0}^2-12r_{0}^3)}=\alpha,\]
where $r_{0}$ is the real root of \eqref{eq4.3} in $(0,1)$. Theorem \ref{th4.6} gives

\begin{corollary}
If $f \in \mathcal{S}_{H}^{*0}$ and $g \in \mathcal{K}_{H}^{0}$, then $f *g$ is univalent and fully starlike of order $\alpha$ in the disk $|z|<r_{0}$, where $r_{0}=r_{0}(\alpha)$ is the real root of \eqref{eq4.3} in $(0,1)$.
\end{corollary}

\end{document}